\pdfoutput=1
\documentclass[a4paper,reqno]{amsart}

\usepackage{etoolbox}

\newtoggle{arxiv}\toggletrue{arxiv} 
\newtoggle{birk}\togglefalse{birk}
\newtoggle{birk_t2}\togglefalse{birk_t2}
\newtoggle{ams_proc}\togglefalse{ams_proc}
\newtoggle{siam}\togglefalse{siam}

\newtoggle{boldproof}\togglefalse{boldproof} 

\toggletrue{arxiv}
\toggletrue{boldproof}

\usepackage[utf8]{inputenc}
\usepackage[english]{babel}
\usepackage{microtype} 
\usepackage[T1]{fontenc}

\usepackage{hyperref} 
\hypersetup{hidelinks}

\usepackage{amsmath,amssymb,amsthm}

\usepackage[nameinlink]{cleveref} 

\usepackage{thmtools, thm-restate} 

\usepackage{mathtools} 
\usepackage{upgreek}
\usepackage{orcidlink}

\usepackage{tikz} 
\usetikzlibrary{cd}
\usetikzlibrary{calc}
\usetikzlibrary{intersections}
\usetikzlibrary{patterns.meta}

\usepackage{multirow}

\usepackage{dsfont} 
\usepackage{pifont} 
\usepackage{nicefrac} 

\usepackage{ifthen}
\usepackage{enumitem}
\setlist[itemize,1]{leftmargin=20pt}

\newlist{steps}{enumerate}{1}
    \setlist[steps]{label=\textup{\textbf{\arabic{*}.~Step:}}, ref=\textup{\arabic{*}.~Step}, align=left, leftmargin=0pt, itemindent=*,labelindent=0pt, labelsep=3pt, itemsep=3pt, parsep=2pt,topsep=2pt}%

\usepackage{framed}

\newcounter{i} 
\newtoks\striche 

\newcommand{\R}{\mathbb{R}}
\newcommand{\N}{\mathbb{N}} 

\newcommand{\argdot}{\boldsymbol{\cdot}}

\newcommand{\Lp}[1]{\mathrm{L}^{#1}} 
\newcommand{\Lptan}[1]{\mathrm{L}_{\tau}^{#1}} 

\newcommand{\indicator}{\mathds{1}} 

\newcommand{\cl}[2][]{\overline{#2}\ifthenelse{ \equal{#1}{} }{}{^{#1}}} 

\newcommand{\ball}{\mathrm{B}}

\DeclareMathOperator{\spn}{span}

\DeclareMathOperator{\supp}{supp}
\DeclareMathOperator{\dom}{dom}

\newcommand{\grad}{\nabla}
\newcommand{\tangrad}{\nabla_{\tau}}
\DeclareMathOperator{\rot}{curl}

\DeclareMathOperator{\dist}{dist}

\DeclarePairedDelimiter{\set}{\{}{\}}
\DeclarePairedDelimiter{\norm}{\lVert}{\rVert}
\DeclarePairedDelimiter{\abs}{\vert}{\vert}

\DeclarePairedDelimiterX{\dset}[2]{\{}{\}}{#1\,\delimsize\vert\,\mathopen{} #2}
\DeclarePairedDelimiterX{\scprod}[2]{\langle}{\rangle}{#1,#2}
\DeclarePairedDelimiterX{\dualprod}[2]{\langle}{\rangle}{#1,#2}
\DeclarePairedDelimiterX{\sdprod}[2]{\llangle}{\rrangle}{#1,#2} 

\newcommand{\adjunsymb}{\ast} 

\newcommand{\adjun}[1][1]{%
  \setcounter{i}{1}%
  \striche={\adjunsymb}%
  \loop%
  \ifnum\value{i}<#1%
  \striche=\expandafter{\the\expandafter\striche\adjunsymb}%
  \stepcounter{i}%
  \repeat%
  ^{\the\striche}%
}

\newcommand{\mapping}[4]{%
  \left\{%
    \begin{array}{rcl}%
      #1 &\to & #2, \\
      #3 &\mapsto & #4
    \end{array}%
  \right.%
}

\newcommand{\transposed}{^{\mathsf{T}}}
\newcommand{\trans}{\transposed}

\newcommand{\soboH}{\mathrm{H}}

\newcommand{\idop}{\mathrm{I}}

\newcommand{\hH}{\hat{\soboH}}
\newcommand{\cH}{\mathring{\soboH}}

\newcommand{\conC}{\mathrm{C}}
\newcommand{\Cc}[1][\infty]{\mathring{\mathrm{C}}^{#1}}

\newcommand{\boundtr}[1][]{\gamma_{0}\ifthenelse{\equal{#1}{}}{}{\big\vert_{#1}}}
\newcommand{\normaltr}[1][]{\gamma_{\nu}\ifthenelse{\equal{#1}{}}{}{\big\vert_{#1}}}
\newcommand{\tantr}[1][]{\pi_{\tau}\ifthenelse{\equal{#1}{}}{}{\big\vert_{#1}}}
\newcommand{\tanxtr}[1][]{\gamma_{\tau}\ifthenelse{\equal{#1}{}}{}{\big\vert_{#1}}}

\theoremstyle{plain}
\newtheorem{theorem}{Theorem}[section]
\newtheorem{lemma}[theorem]{Lemma}

\newtheorem{corollary}[theorem]{Corollary}

\theoremstyle{definition}
\newtheorem{definition}[theorem]{Definition}

\newtheorem{remark}[theorem]{Remark}

\makeatletter

    \providecommand{\proofNameStyle}{\bfseries}
    \renewenvironment{proof}[1][\proofname]{\par
      \pushQED{\qed}%
      \normalfont \topsep6\p@\@plus6\p@\relax
      \trivlist
      \item[\hskip\labelsep\proofNameStyle
      #1\@addpunct{.}]\ignorespaces
    }{%
      \popQED\endtrivlist\@endpefalse
    }

\makeatother

\begin{document}

\title[Weak equals strong L\textsuperscript{2} regularity]{Weak equals strong L\textsuperscript{2} regularity for partial tangential traces on Lipschitz domains}

\ifboolexpr{togl{birk_t2} or togl{birk}}{%
  }{%
  }%

\author[N.~Skrepek]{Nathanael Skrepek\,\orcidlink{0000-0002-3096-4818}}
    \thanks{\textit{E-mail}: \href{mailto:academia@skrepek.at}{academia@skrepek.at}}
  
\address{%
  Department of Applied Mathematics,
  University of Twente,
  P.O. Box 217,
  7500 AE Enschede,
  The Netherlands
  }
\email{academia@skrepek.at}

\author[D.~Pauly]{Dirk Pauly\,\orcidlink{0000-0003-4155-7297}}
\address{Technische Universität Dresden (TUDD), Fakultät Mathematik, Institut für Analysis, Zellescher Weg 12-14, 01069 Dresden, Germany}
\email{dirk.pauly@tu-dresden.de}

\begin{abstract}
  We investigate the boundary trace operators that naturally correspond to $\soboH(\rot,\Omega)$, namely the tangential and twisted tangential trace, where $\Omega \subseteq \R^{3}$. In particular we regard partial tangential traces, i.e., we look only on a subset $\Gamma$ of the boundary $\partial\Omega$.
  We assume both $\Omega$ and $\Gamma$ to be strongly Lipschitz (possibly unbounded).
  We define the space of all $\soboH(\rot,\Omega)$ fields that possess a $\Lp{2}$ tangential trace in a weak sense and show that the set of all smooth fields is dense in that space, which is a generalization of \cite{BeBeCoDa97}.
  This is especially important for Maxwell's equation with mixed boundary condition as we answer the open problem by Weiss and Staffans in \cite[Sec.~5]{WeSt13} for strongly Lipschitz pairs.
\end{abstract}

\maketitle

\section{Introduction}

We will regard a strongly Lipschitz domain $\Omega \subseteq \R^{3}$ and the Sobolev space that corresponds to the $\rot$ operator
\[
  \soboH(\rot,\Omega) = \dset{f \in \Lp{2}(\Omega)}{\rot f \in \Lp{2}(\Omega)}
\]
and the ``natural'' boundary traces that are associated with the $\rot$ operator
\begin{align*}
  \tantr f \coloneqq \bigl(\nu \times f \big\vert_{\partial\Omega}\bigr) \times \nu \quad\text{and}\quad \tanxtr f \coloneqq \nu \times f \big\vert_{\partial\Omega}
  \quad\text{for}\quad f \in \Cc(\R^{3}),
\end{align*}
where $\nu$ denotes the outer normal vector on the boundary of $\Omega$ and $\Cc(\R^{3})$ denotes the $\conC^{\infty}$ functions with compact support on $\R^{3}$. These boundary traces are called \emph{tangential trace} and \emph{twisted tangential trace}, respectively.
They are motivated by the integration by parts formula
\begin{equation*}
  \scprod{\rot f}{g}_{\Lp{2}(\Omega)} - \scprod{f}{\rot g}_{\Lp{2}(\Omega)} = \scprod{\tanxtr f}{\tantr g}_{\Lp{2}(\partial\Omega)}.
\end{equation*}
We could even extend these boundary operators to $\soboH(\rot,\Omega)$ by introducing suitable boundary spaces, see e.g., \cite{BuCoSh02} for full boundary traces or \cite{Sk21} for partial boundary traces. However, in this article we focus on those $f \in \soboH(\rot,\Omega)$ that have a meaningful $\Lp{2}(\partial\Omega)$ (twisted) tangential trace.
Hence, for $\Gamma \subseteq \partial\Omega$ we are interested in the following spaces
\begin{align*}
  \cH_{\Gamma}(\rot,\Omega) &= \dset{f \in \soboH(\rot,\Omega)}{\tantr f = 0 \;\text{on}\; \Gamma}, \\
  \hH_{\Gamma}(\rot,\Omega) &= \dset{f \in \soboH(\rot,\Omega)}{\tantr f \;\text{is in}\; \Lp{2}(\Gamma)}.
\end{align*}
where we  will later state precisely what we mean by $\tantr f = 0$ on $\Gamma$ and $\tantr f \in \Lp{2}(\Gamma)$. In particular we are interested in $\hH_{\Gamma}(\rot,\Omega)$. Similar to Sobolev spaces there are two approaches to $\tantr f \in \Lp{2}(\Gamma)$: A weak approach by representation in an inner product and a strong approach by limits of regular functions.
We use the weak approach as definition, see \Cref{def:weak-tangential-trace}.
The question that immediately arises is:
\begin{center}
  \textit{``Do both approaches lead to the same space?''}
\end{center}

In~\cite[eq.\ (5.20)]{WeSt13} the authors observed this problem and concluded that it can cause ambiguity for boundary conditions, if the approaches don't coincide. In fact they stated this issue at the end of section~5 in \cite{WeSt13} as an open problem.
This problem can actually be viewed as a more general question that arises for quasi Gelfand triples, see~\cite[Conjecture~6.7]{Sk24}.

We will not explicitly define the strong approach, but show that the ``most'' regular functions ($\conC^{\infty}$ functions) are already dense in the weakly defined space, which immediately implies that any strong approach with less regular functions (e.g., $\soboH^{1}$) will lead to the same space. This is exactly what was done in \cite{BeBeCoDa97} for bounded $\Omega$ with $\Gamma = \partial\Omega$. Hence, we present a generalization of \cite{BeBeCoDa97} for \emph{partial} $\Lp{2}$ tangential traces and \emph{unbounded} $\Omega$.
In particular, we aim to prove the following two main theorems.

\begin{theorem}
  Let $\Omega$ be a (possibly unbounded) strongly Lipschitz domain and $\Gamma_{1} \subseteq \partial\Omega$ such that $(\Omega,\Gamma_{1})$ is a strongly Lipschitz pair, then
  $\Cc(\R^{3})$ is dense in $\hH_{\Gamma_{1}}(\rot,\Omega)$ with respect to $\norm{\argdot}_{\hH_{\Gamma_{1}}(\rot,\Omega)}$.
\end{theorem}

\begin{theorem}
  Let $\Omega$ be a (possibly unbounded) strongly Lipschitz domain and $\Gamma_{0} \subseteq \partial\Omega$ such that $(\Omega,\Gamma_{0})$ is a strongly Lipschitz pair, then
  $\Cc_{\Gamma_{0}}(\R^{3})$ is dense in
  $\hH_{\partial\Omega}(\rot,\Omega) \cap \cH_{\Gamma_{0}} (\rot,\Omega)$
  with respect to $\norm{\argdot}_{\hH_{\partial\Omega}(\rot,\Omega)}$.
\end{theorem}

However, it turned out that it is best to prove them in reversed order.

The importance of our density results lies in the context of Maxwell's equations with boundary conditions that involve a mixture of $\tantr$ and $\tanxtr$ in the sense of linear combination, e.g., this simplified instance of Maxwell's equations
\begin{alignat*}{2}
  \partial_{t} E(t,\zeta) &= \rot H(t,\zeta), \qquad&&t \geq 0, \zeta \in \Omega,\\
  \partial_{t} H(t,\zeta) &= -\rot E(t,\zeta), \qquad&&t \geq 0, \zeta \in \Omega,\\
  \tantr E(t,\xi) + \tanxtr H(t,\xi) &= 0, \qquad&&t \geq 0, \xi \in \Gamma_{1},\\
  \tantr E(t,\xi) &= 0, \qquad&&t \geq 0, \xi \in \Gamma_{0}.
\end{alignat*}
In order to properly formulate the boundary conditions we need to know what functions $E$, $H$ have tangential traces that allow such a linear combination. Especially when it comes to well-posedness our density results are needed to avoid the ambiguity that was observed in \cite{WeSt13}.

As suspected by Weiss and Staffans in \cite{WeSt13} the regularity of the interface of $\Gamma_{0} \subseteq \partial\Omega$ and $\Gamma_{1} \coloneqq \partial\Omega \setminus \cl{\Gamma_{0}}$ seems to play a role. At least for our answer we need that the boundary of $\Gamma_{0}$ is also strongly Lipschitz.

In particular our strategy is based on the following decomposition from \cite[Thm.~5.2]{PaScho22} for bounded $\Omega$
\begin{equation}%
  \label{eq:decomposition-of-cH-Gamma-0}
  \cH_{\Gamma_{0}}(\rot,\Omega) = \cH^{1}_{\Gamma_{0}}(\Omega) + \grad \cH_{\Gamma_{0}}^{1}(\Omega)
\end{equation}
which requires $(\Omega,\Gamma_{0})$ to be a strongly Lipschitz pair.
For our main result we consider the intersection of $\hH_{\partial\Omega}(\rot,\Omega)$ and \eqref{eq:decomposition-of-cH-Gamma-0}:
\begin{equation*}
  \hH_{\partial\Omega}(\rot,\Omega) \cap \cH_{\Gamma_{0}}(\rot,\Omega)
  = \cH^{1}_{\Gamma_{0}}(\Omega)
  + \hH_{\partial\Omega}(\rot,\Omega) \cap \grad \cH_{\Gamma_{0}}^{1}(\Omega).
\end{equation*}
Every element of $\cH^{1}_{\Gamma_{0}}(\Omega)$ can be approximated by a sequence in $\Cc_{\Gamma_{0}}(\R^{3})$ w.r.t.\ $\norm{\argdot}_{\soboH^{1}(\Omega)}$ (see \cite[Lem.~3.1]{BaPaScho16}), which is a stronger norm than the ``natural'' norm of $\hH_{\partial\Omega}(\rot,\Omega)$.
Hence, the challenging part will be finding an approximation by $\Cc_{\Gamma_{0}}(\R^{3})$ elements for all elements in
\[
  \hH_{\partial\Omega}(\rot,\Omega) \cap \grad \cH^{1}_{\Gamma_{0}}(\Omega).
\]

It even turned out that, if we can prove the decomposition~\eqref{eq:decomposition-of-cH-Gamma-0} also for less regular $\Gamma_{0}$, then our main theorems would automatically generalize for those less regular partitions of $\partial\Omega$, since this is the only occasion where the regularity of $\Gamma_{0}$ is used.

Finally, we will conclude the results for unbounded $\Omega$ by localization.

\section{Preliminary}%
\label{sec:preliminary}

For $\Omega \subseteq \R^{d}$ open and $\Gamma \subseteq \partial\Omega$ open we use the following notation (as in \cite{BaPaScho16})
\begin{align*}
  \Cc(\Omega) &\coloneqq \dset[\big]{f \in \conC^{\infty}(\Omega)}{\supp f \ \text{is compact in}\ \Omega}, \\
  \Cc_{\Gamma}(\Omega) &\coloneqq \dset*{f\big\vert_{\Omega}}{f \in \Cc(\R^{d}), \dist(\Gamma,\supp f) > 0},
\end{align*}
and $\soboH^{1}(\Omega)$ denotes the usual Sobolev space and $\cH^{1}_{\Gamma}(\Omega)$ is the subspace of $\soboH^{1}(\Omega)$ with homogeneous boundary data on $\Gamma$, i.e., $\cH_{\Gamma}^{1}(\Omega) = \cl[\soboH^{1}(\Omega)]{\Cc_{\Gamma}(\Omega)}$.

Note that the trace operators $\tantr$ and $\tanxtr$ are called \emph{tangential} traces, because $\nu \cdot \tantr f = 0$ and $\nu \cdot \tanxtr f = 0$. Hence, it is natural to introduce the \emph{tangential $\Lp{2}$} space on $\Gamma \subseteq \partial\Omega$ by
\begin{equation*}
  \Lptan{2}(\Gamma) = \dset{f \in \Lp{2}(\Gamma)}{\nu \cdot f = 0}.
\end{equation*}
This space is again a Hilbert space with the $\Lp{2}(\Gamma)$ inner product. Moreover, both $\tantr \Cc_{\partial\Omega \setminus \Gamma}(\R^{3})$ and $\tanxtr \Cc_{\partial\Omega \setminus \Gamma}(\R^{3})$ are dense in that space.

Next we recall the definition of a strongly Lipschitz domain, see e.g., \cite{grisvard1985}. Note that we use the definition also for unbounded domains.\footnote{We do not ask for a uniform Lipschitz constant or similar modifications for unbounded domains} Moreover, we need $\soboH^{1}$ spaces on strongly Lipschitz boundaries, see e.g., \cite{Sk25a} for a careful treatment.

\begin{definition}\label{def:strong-lipschitz-domain}
  Let $\Omega$ be an open subset of $\R^{d}$. We say $\Omega$ is a \emph{strongly Lipschitz domain}, if for every $p \in \partial\Omega$ there exist $\epsilon,h > 0$, a hyperplane $W = \spn\set{w_{1},\dots,w_{d-1}}$, where $\set{w_{1},\dots,w_{d-1}}$ is an orthonormal basis of $W$, and a Lipschitz continuous function $a_{W}\colon (p + W) \cap \ball_{\epsilon}(p) \to (-\frac{h}{2},\frac{h}{2})$ such that
  \begin{align*}
    \partial\Omega \cap C_{\epsilon,h}(p) &= \dset{x + a_{W}(x)v}{x \in (p+W) \cap \ball_{\epsilon}(p)}, \\
    \Omega \cap C_{\epsilon,h}(p) &= \dset{x + sv}{x \in (p+W) \cap \ball_{\epsilon}(p), -h < s < a_{W}(x)},
  \end{align*}
  where $v$ is the normal vector of $W$ and $C_{\epsilon,h}(p)$ is the cylinder $\dset{x + \delta v}{x \in (p+ W) \cap \ball_{\epsilon}(p), \delta \in (-h,h)}$.

  The boundary $\partial\Omega$ is then called \emph{strongly Lipschitz boundary}.
\end{definition}

The hyperplane $W$ and the vector $v$ induce a new coordinate system centered in $p$. With respect to this coordinate system the boundary of $\Omega$ is locally given by the graph of a Lipschitz continuous function, see \Cref{fig:lipschitz-boundary}.
\begin{figure}[h]
  \newcommand{\nvec}{v}
  \centering
  \begin{tikzpicture}[scale=2]
    \coordinate (vec-n) at (-1,0.5); 
    \coordinate (vec-t) at (0.5,1); 

    \coordinate[label=above:$p$] (p) at (0.5,1);
    \coordinate[label=below:$\nvec$] (n) at ($(p) + 0.25*(vec-n)$);

    \coordinate (H1) at ($(p) - (vec-t)$);
    \coordinate[label=left:$W$] (H2) at ($(p) + (vec-t)$);

    \coordinate (C1) at ($(p) + 0.6*0.89442719*(vec-t) + 0.7*(vec-n)$);
    \coordinate (C2) at ($(p) + 0.6*0.89442719*(vec-t) - 0.7*(vec-n)$);
    \coordinate (C3) at ($(p) - 0.6*0.89442719*(vec-t) - 0.7*(vec-n)$);
    \coordinate (C4) at ($(p) - 0.6*0.89442719*(vec-t) + 0.7*(vec-n)$);


    \draw (H1) -- (H2);

    \tikzstyle{transformation} = [shift={(-3,1)},rotate=-63.43]

    \draw ([transformation]p) + (0,0.05) -- ++(0,-0.05);
    \node[shift={(0,0.3)}] at ([transformation]p) {$0$};
    \draw ([transformation]C1) -- ([transformation]C2) -- ([transformation]C3) -- ([transformation]C4) -- ([transformation]C1);
    \draw ([transformation]H1) -- ([transformation]H2);
    \begin{scope}
      \clip ([transformation]C1) -- ([transformation]C2) -- ([transformation]C3) -- ([transformation]C4) -- ([transformation]C1);
      \draw[transformation,very thick] (1,0) .. controls (0.48,0.18) .. ([transformation]p);
      \draw[transformation,very thick] ([transformation]p) .. controls (1,1.7) .. (2,2);
    \end{scope}

    \draw (C1) -- (C2) -- (C3) -- (C4) -- (C1);
    \node[above] at (C1) {$C_{\epsilon,h}(p)$};

    \draw[->] (p) -- (n); 

    \draw (1,0) .. controls (0.48,0.18) .. (p);
    \draw (p) .. controls (1,1.7) .. (2,2);
    \draw
    (2,2) to [out=0,in=110] (3,1)
    (3,1) to [out=-100,in=5] (1.5,0.3)
    (1.5,0.3) to (1,0);

    \begin{scope}
      \clip (C1) -- (C2) -- (C3) -- (C4) -- (C1);
      \draw[very thick] (1,0) .. controls (0.48,0.18) .. (p);
      \draw[very thick] (p) .. controls (1,1.7) .. (2,2);
    \end{scope}

    \node at (1.75,1.05) {$\Omega$};
    \filldraw (p) circle (0.1em); 
  \end{tikzpicture}
  \caption{Lipschitz boundary}
  \label{fig:lipschitz-boundary}
\end{figure}

Corresponding to a strongly Lipschitz domain we define the following bi-Lipschitz continuous mapping
\begin{align*}
  k\colon
  \mapping{\partial\Omega \cap C_{\epsilon,h}(p)}{\ball_{\epsilon}(0) \subseteq \R^{d-1}}{\zeta}{W\transposed (\zeta - p),}
\end{align*}
where we used $W$ as the matrix \(\begin{bsmallmatrix} w_{1} & \dots & w_{d-1} \end{bsmallmatrix}\).
We call this mapping a \emph{strongly Lipschitz chart} of $\partial\Omega$ and we call its domain the \emph{chart domain}. Its inverse is given by
\begin{align*}
  k^{-1}\colon
  \mapping{\ball_{\epsilon}(0) \subseteq \R^{d-1}}{\partial\Omega \cap C_{\epsilon,h}(p)}{x}{p + Wx + a(x)v,}
\end{align*}
where we define $a(x) \coloneqq a_{W}(p+Wx)$, which is then a Lipschitz continuous function from $\ball_{\epsilon}(0) \subseteq \R^{d-1}$ to $\R$. Charts are used to regard the surface of $\Omega$ locally as a flat subset of $\R^{d-1}$. Every restriction of a chart $k$ to an open $\Gamma \subseteq \partial\Omega$ is again a chart.

\begin{definition}
  Let $\Omega$ be a strongly Lipschitz domain in $\R^{d}$. Then we say that an open $\Gamma_{0} \subseteq \partial\Omega$ is strongly Lipschitz, if for every $p \in \cl{\Gamma_{0}}$ there exists a chart $k\colon \partial\Omega \cap C_{\epsilon,h}(p) \to \ball_{\epsilon}(0) \subseteq \R^{d-1}$ such that $k(\Gamma_{0})$ is a strongly Lipschitz domain in $\R^{d-1}$.

  The boundary $\partial \Gamma_{0}$ is then called \emph{strongly Lipschitz boundary}.
\end{definition}

Note that it is sufficient to reduce the condition in the previous definition to $p \in \partial \Gamma_{0}$ instead of $p \in \cl{\Gamma_{0}}$.


\begin{definition}
  We call $(\Omega,\Gamma_{0})$ a \emph{strongly Lipschitz pair}, if $\Omega$ is a strongly Lipschitz domain and $\Gamma_{0} \subseteq \partial\Omega$ is strongly Lipschitz.
\end{definition}

Note that if $\Gamma_{0} \subseteq \partial\Omega$ is strongly Lipschitz, then also $\Gamma_{1} \coloneqq \partial\Omega \setminus \cl{\Gamma_{0}}$ is strongly Lipschitz. Hence, if $(\Omega,\Gamma_{0})$ is a strongly Lipschitz pair, then also $(\Omega, \Gamma_{1})$ is.

\begin{framed}
  Since we only deal with strongly Lipschitz domains and boundaries, we will omit the term ``strongly'' and just say \emph{Lipschitz domain}, \emph{Lipschitz boundary} and \emph{Lipschitz chart}.
\end{framed}

Recall the definition of a $\soboH^{1}$ function on the boundary of a Lipschitz domain, see e.g., \cite{Sk25a}. Since we only need the following for bounded domains we will state the following definition only for bounded domains.
\begin{definition}
  Let $\Omega \subseteq \R^{d}$ be a bounded Lipschitz domain.
  We say $f \in \Lp{2}(\partial\Omega)$ is in $\soboH^{1}(\partial\Omega)$, if for every Lipschitz chart $k\colon \Gamma \to U$ the mapping
  \begin{equation*}
    f \circ k^{-1} \quad \text{is in}\quad \soboH^{1}(U).
  \end{equation*}
\end{definition}

\section{Density results for $W(\Omega)$}

In this section we restrict ourselves to bounded Lipschitz domains $\Omega$. However, this is mostly for convenience.

\begin{definition}
  Let $\Omega \subseteq \R^{d}$ be a bounded Lipschitz domain.
  Then we define
  \begin{align*}
    W(\Omega)
    &\coloneqq \dset*{f \in \soboH^{1}(\Omega)}{f \big\vert_{\partial\Omega} \in \soboH^{1}(\partial\Omega)},
    \\
    \norm{f}_{W(\Omega)}
    &\coloneqq \Big(\norm{f}_{\soboH^{1}(\Omega)}^{2} + \norm{f \big\vert_{\partial\Omega}}_{\soboH^{1}(\partial\Omega)}^{2} \Big)^{\nicefrac{1}{2}}.
  \end{align*}
\end{definition}

The next lemma a is a crucial tool in our construction. The basic idea is: Take a smooth function $f$ with compact support on a flat domain ($U \subseteq \R^{d-1}$) extend it on the entire hyperplane $\R^{d-1}$ by $0$, and then extend it constantly in the orthogonal direction, i.e., $F\left(\begin{psmallmatrix} \zeta \\ \lambda \end{psmallmatrix} \right) = f(\zeta)$, where $\zeta \in \R^{d-1}$ and $\lambda \in \R$.
Multiplying with a cutoff function $\chi$ makes sure that this extension has compact support. By rotation and translation this can be done for arbitrary hyperplanes. \Cref{fig:illustration-of-extension} illustrates the construction.

\begin{lemma}\label{le:extend-from-boundary-to-volume}
  Let $\Omega\subseteq \R^{d}$ be a bounded Lipschitz domain, $\Gamma \subseteq \partial\Omega$,
  $k\colon \Gamma \to U$ be a Lipschitz chart and $f \in \soboH^{1}(\partial\Omega)$ with compact support in $\Gamma' \subseteq \Gamma$.
  Then there exists an $F \in \soboH^{1}(\R^{d}) \cap W(\Omega) \cap \cH^{1}_{\partial\Omega \setminus \Gamma'}(\Omega)$ such that $F \big\vert_{\partial\Omega} = f$.
  Moreover, there exists a sequence $(F_{n})_{n\in\N}$ in $\Cc_{\partial\Omega \setminus \Gamma'}(\R^{d})$ that converges to $F$ w.r.t.\ $\norm{\argdot}_{\soboH^{1}(\R^{d})} + \norm{\argdot}_{W(\Omega)}$, i.e., $F_{n}$ converges to $F$ in $\soboH^{1}(\R^{d})$ and $F_{n} \big\vert_{\partial\Omega}$ converges to $F \big\vert_{\partial\Omega}$ in $\soboH^{1}(\partial\Omega)$.
\end{lemma}

\begin{proof}
  Let $p$, $W$ and $v$ be the point, hyperplane and normal vector, respectively, to the chart $k$. In particular $k^{-1}$ is given by
  \begin{equation*}
    k^{-1} \colon
    \mapping{U \subseteq \R^{d-1}}{\Gamma}{x}{p + W x + a(x) v,}
  \end{equation*}
  where $U$ is open and $a$ is the Lipschitz function.
  Let $\chi \in \Cc(\R)$ be a cutoff function such that
  \begin{equation*}
    \begin{aligned}[c]
      \chi(\lambda) \in
      \begin{cases}
        \set{1}, & \abs{\lambda} < \nicefrac{3}{2}\norm{a}_{\infty},\\
        [0,1], & \abs{\lambda} \in (\nicefrac{3}{2},2)\norm{a}_{\infty}, \\
        \set{0}, & \abs{\lambda} > 2\norm{a}_{\infty}.
      \end{cases}
    \end{aligned}
    \quad
    \begin{aligned}[c]
      \begin{tikzpicture}
        \coordinate (A) at (2.2,0);
        \coordinate (B) at (-2.2,0);
        \coordinate (C) at (1.3,0);
        \coordinate (D) at (-1.3,0);

        \draw[Latex-Latex] ($(B) + (-1,0)$) -- ($(A) + (1,0)$);
        \draw[-Latex] (0,-0.2) -- (0,1.5);

        \draw (A)++(0,-0.1) -- ($(A)+(0,0.1)$);
        \draw (B)++(0,-0.1) -- ($(B)+(0,0.1)$);

        \draw (C)++(0,-0.1) -- ($(C)+(0,0.1)$);
        \draw (D)++(0,-0.1) -- ($(D)+(0,0.1)$);

        \draw (-0.1,1) -- (0.1,1);
        \node[above] at (0.2,1) {$1$};

        \draw[thick] (D)++(0,1) -- ($(C)+(0,1)$);

        \draw[thick] (D)++(0,1) to[out=180, in=0] (B);
        \draw[thick] (C)++(0,1) to[out=0, in=180] (A);

        \draw[thick] (A)++(0.5,0) -- (A);
        \draw[thick] (B)++(-0.5,0) -- (B);

        \node[shift={(0.3,0)},below] at (A) {$2\norm{a}_{\infty}$};

        \node[shift={(-0.3,0)},below] at (C) {$\nicefrac{3}{2}\norm{a}_{\infty}$};

        \node[above] at ($(D) + (-0.2,1)$) {$\chi$};
      \end{tikzpicture}
    \end{aligned}
  \end{equation*}
  By definition $\hat{f} = f \circ k^{-1}$ is in $\soboH^{1}(U)$ and since $f$ has compact support in $\Gamma'$ we conclude $\hat{f} \in \cH^{1}(U)$ with support in $U' \coloneqq k(\Gamma')$
  Note that we can extend $\hat{f} \in \cH^{1}(U)$ on $\R^{d}$ by $0$. We define
  \begin{equation*}
    F(\zeta) = \chi(v \cdot (\zeta - p)) \hat{f}\big(W\trans (\zeta - p)\big) \quad\text{for}\quad \zeta \in \R^{d}
  \end{equation*}
  The support of $F$ is inside of a rotated and translated
  version of $U' \times \supp \chi$, in particular
  \begin{equation*}
    \supp F \subseteq p + \begin{bmatrix} W & v\end{bmatrix} (U' \times \supp \chi) \eqqcolon \Xi.
  \end{equation*}

  \begin{figure}[h]
    \centering
    \begin{tikzpicture}
      \colorlet{mygreen}{green!40!black}
      \coordinate (p) at (0,0);
      \coordinate (leftp) at (-1.5,-1);
      \coordinate (rightp) at (1.5,1);

      \coordinate (leftsupp) at (-0.75,-0.6);
      \coordinate (rightsupp) at (1,0.7);

      \coordinate (height) at (0,2.5);

      \tikzstyle{transformation} = []

      \draw ([transformation]$(leftp|-p) - (height)$) -- ([transformation]$(leftp|-p) + (height)$) -- ([transformation]$(rightp|-p) + (height)$) -- ([transformation]$(rightp|-p) - (height)$) -- cycle;

      \draw ([transformation]$(p-|leftp) - (0.5,0)$) -- ([transformation]$(p-|rightp) + (0.5,0)$);
      \filldraw ([transformation]p) circle (0.05);
      \node[above,xshift=4] at ([transformation]p) {$0$};

      \draw[transformation]
      ([transformation]leftp)
      to ([transformation]leftsupp)
      to[out=0,in=210] ([transformation]p)
      to[out=-10,in=230] ([transformation]rightsupp)
      to[out=50,in=170] ([transformation]rightp);

      \draw[very thick,mygreen] ($(leftsupp|-p)$) -- ($(rightsupp|-p)$);
      \node[below,mygreen] at ($(p) + (0.6,0)$) {$\supp \hat{f}$};

      \draw[very thick,blue] ($(p) + (height)$) -- ($(p) - (height)$);
      \node[blue,left] at ($(p) + 0.5*(height)$) {$\supp \chi$};

      \tikzstyle{transformation} = [shift={(6.5,1.5)},rotate=-45]

      \draw[transformation]
      ([transformation]leftp)
      to ([transformation]leftsupp)
      to[out=0,in=210] ([transformation]p)
      to[out=-10,in=230] ([transformation]rightsupp)
      to[out=50,in=170] ([transformation]rightp)
      to[out=-10,in=100] (5,-1)
      to[out=240,in=10] (3.5,-2.5)
      to[out=230,in=0] (1,-4) -- (-4,-2) --(-3,0) -- cycle;

      \node at ([transformation]$0.4*(p) + 0.7*(3.5,-2.5)$) {$\Omega$};
      \node[left] at ([transformation]$(1,-4)$) {$\partial\Omega$};

      \draw[very thick,mygreen,transformation]
      ([transformation]leftsupp)
      to[out=0,in=210] ([transformation]p)
      to[out=-10,in=230] ([transformation]rightsupp);

      \filldraw ([transformation]p) circle (0.05);
      \node[above] at ([transformation]p) {$p$};

      \node[right,mygreen] at ($([transformation]p) + (0.15,0.1)$) {$\supp f$};

      \draw ([transformation]$(leftp|-p) - (height)$) -- ([transformation]$(leftp|-p) + (height)$) -- ([transformation]$(rightp|-p) + (height)$) -- ([transformation]$(rightp|-p) - (height)$) -- cycle;
      \node[right] at ([transformation]$(rightp|-p) + (height)$) {$C_{\epsilon,h}(p)$};

      \draw ([transformation]$(p-|leftp) - (0.5,0)$) -- ([transformation]$(p-|rightp) + (0.5,0)$);
      \node[right] at ([transformation]$(p-|rightp) + (0.5,0)$) {$W$};

      \draw[very thick,blue] ([transformation]$(p) + (height)$) -- ([transformation]$(p) - (height)$);
      \node[blue,left] at ([transformation]$(p) + 0.5*(height)$) {``$\supp \chi$''};

    \end{tikzpicture}
    \caption{Illustration of the construction of \Cref{le:extend-from-boundary-to-volume}}%
    \label{fig:illustration-of-extension}
  \end{figure}
  Note that by construction of $\chi$ we have $\supp F \cap \partial\Omega \subseteq \Gamma'$, therefore $F \big\vert_{\partial\Omega \setminus \Gamma'} = 0$.
  Since $\hat{f} \in \soboH^{1}(\R^{d-1})$ it is straight forward that $F$ possess $\Lp{2}(\R^{d})$ directional derivatives in $W$ directions. Moreover, by construction (and the Leibniz product rule) $\tfrac{\partial}{\partial v} F = \chi' \hat{f}(W\trans (\argdot-p))$, which implies $F \in \soboH^{1}(\R^{d})$.
  By definition of a Lipschitz chart we have $\abs{v \cdot (\zeta - p)} \leq \norm{a}_{\infty}$ for $\zeta \in \Gamma$ and hence
  \[
    F(\zeta) = \underbrace{\chi(v  \cdot (\zeta - p))}_{=\mathrlap{1}}\hat{f}(W\trans (\zeta - p)) = \hat{f} \circ k(\zeta) = f(\zeta)
    \quad\text{for}\quad \zeta \in \Gamma
  \]
  (a.e.\ w.r.t.\ the surface measure).

  By assumption on $\hat{f}$ there exists a sequence $(\hat{f}_{n})_{n\in\N}$ in $\Cc(U)$ that converges to $\hat{f}$ w.r.t.\ $\norm{\argdot}_{\soboH^{1}(U)}$. Note that $\hat{f}_{n}$ is also in $\Cc(\R^{d-1})$.
  We define
  \begin{align*}
    F_{n}(\zeta) = \chi(v \cdot (\zeta - p)) \hat{f}_{n}\big(W\trans (\zeta - p)\big)  \quad\text{for}\quad \zeta \in \R^{d}
  \end{align*}
  Note that $F_{n}$ is the composition of $\conC^{\infty}$ mappings and therefore in $\conC^{\infty}(\R^{d})$. Again, the support of $F_{n}$ is contained in the bounded set
  $\Xi$
  and therefore compact, which implies $F_{n} \in \Cc(\R^{d})$.
  Note that $F_{n} \circ k^{-1} = \hat{f}_{n}$, which implies $(F_{n} \circ k^{-1})_{n\in\N}$ converges to $\hat{f}$ w.r.t.\ $\norm{\argdot}_{\soboH^{1}(U)}$. Since $F_{n} \big\vert_{\partial\Omega \setminus \Gamma} = 0 = F \big\vert_{\partial\Omega \setminus \Gamma}$ we conclude $F_{n}\big\vert_{\partial\Omega} \to F \big\vert_{\partial\Omega}$ in $\soboH^{1}(\partial\Omega)$.
  Finally,
  \begingroup%
  \newcommand{\tmpEq}{%
    \norm{F_{n} - F}_{\soboH^{1}(\R^{3})} \leq \norm{\chi'}_{\infty}\norm{(\hat{f}_{n} - \hat{f})(W\trans(\argdot - p))}_{\soboH^{1}(\Xi)}
    \\%
    \leq 2\norm{a}_{\infty}\norm{\chi'}_{\infty} \norm{\hat{f}_{n} - \hat{f}}_{\soboH^{1}(U)} \to 0}%
    \begin{multline*}
      \tmpEq
    \end{multline*}
  \endgroup%
  finishes the proof.
\end{proof}

We will formulate a generalization of \cite[2.~Preliminaries]{BeBeCoDa97}.

\begin{theorem}\label{th:test-functions-dense-in-W}
  Let $\Omega \subseteq \R^{d}$ be a bounded Lipschitz domain and $\Gamma \subseteq \partial\Omega$ open (in $\partial\Omega$). Then
  $\Cc_{\Gamma}(\R^{d})$ is dense in $W(\Omega) \cap \cH^{1}_{\Gamma}(\Omega)$ w.r.t.\ $\norm{\argdot}_{W(\Omega)}$.
\end{theorem}

\begin{proof}
  Since $\Omega$ is a Lipschitz domain we find for every $p \in \partial\Omega$ a cylinder $C_{\epsilon,h}(p)$ ($\epsilon$ and $h$ depend on $p$) and a Lipschitz chart $k \colon \partial\Omega \cap C_{\epsilon,h}(p) \to \ball_{\epsilon}(0) \subseteq \R^{d-1}$.

  Hence we can cover $\partial \Omega$ by $\bigcup_{p \in \partial\Omega} C_{\epsilon,h}(p)$ and by compactness of $\partial\Omega$ there are already finitely many $p_{i}$, $i \in \set{1,\dots m}$ such that
  \begin{equation*}
    \partial\Omega \subseteq \bigcup_{i=1}^{m} \underbrace{C_{\epsilon_{i},h_{i}}(p_{i})}_{\eqqcolon \mathrlap{C_{i}}}
  \end{equation*}
  We use the short notation $C_{i} \coloneqq C_{\epsilon_{i},h_{i}}(p_{i})$.
  We employ a partition of unity and obtain $(\alpha_{i})_{i=1}^{m}$, subordinate to this cover, i.e.,
  \begin{equation*}
    \alpha_{i} \in \Cc(C_{i}), \quad \alpha_{i}(\zeta) \in [0,1], \quad\text{and}\quad \sum_{i=1}^{m} \alpha_{i}(\zeta) = 1
    \quad\text{for all}\quad \zeta \in \partial\Omega.
  \end{equation*}
  For $f \in W(\Omega) \cap \cH^{1}_{\Gamma}(\Omega)$ we define $f_{i} \coloneqq \alpha_{i} f$. It is straightforward to show $f_{i} \in W(\Omega) \cap \cH^{1}_{\Gamma}(\Omega)$. For every $C_{i}$ there is a Lipschitz chart $k_{i} \colon \Gamma_{i}\to U_{i} \subseteq \R^{d-1}$, where $\Gamma_{i} = \partial\Omega \cap C_{i}$.
  Moreover, $f_{i} \big\vert_{\partial\Omega}$ has support in $\Gamma_{i} \cap \Gamma^{\complement}$, where $\Gamma^{\complement} = (\partial\Omega \setminus \Gamma)$.

  By \Cref{le:extend-from-boundary-to-volume} there is an $F_{i} \in \soboH^{1}(\R^{d}) \cap W(\Omega) \cap \cH_{\partial\Omega\setminus (\Gamma_{i} \cap \Gamma^{\complement})}^{1}(\Omega)$ such that $F_{i} \big\vert_{\partial\Omega} = f_{i} \big\vert_{\partial\Omega}$ and a sequence $(F_{i,n})_{n\in\N}$ in
  $\Cc_{\partial\Omega \setminus (\Gamma_{i} \cap \Gamma^{\complement})}(\R^{d}) \subseteq \Cc_{\Gamma}(\R^{d})$
  that converges to $F_{i}$ in $\soboH^{1}(\R^{d})$ and in $W(\Omega)$. Hence, we have
  \begin{align*}
    f - \sum_{i=1}^{m} F_{i} = \sum_{i=1}^{m} f_{i} - F_{i} \in \cH^{1}(\Omega),
  \end{align*}
  which can be approximated by $(F_{0,n})_{n\in\N}$ in $\Cc(\Omega)$. Hence, $\big(\sum_{i=0}^{m} F_{i,n}\big)_{n\in\N}$ is a sequence in $\Cc_{\Gamma}(\R^{d})$ and converges to $f$ in $W(\Omega)$.
\end{proof}

\section{The bounded case}

\subsection{Density result with homogeneous part}


In this section we will finally define the Sobolev spaces with homogeneous and $\Lp{2}$ partial tangential traces, respectively. Then we will prove one of the main results for bounded domains.

We will use a weak definition for the tangential trace and twisted tangential trace as, e.g., in \cite{PaSk22}.

\begin{definition}\label{def:weak-tangential-trace}
  Let $\Omega$ be a (possibly unbounded) Lipschitz domain and $\Gamma \subseteq \partial\Omega$ open (in $\partial\Omega$).
  \begin{itemize}[itemsep=1ex]
    \item
    We say $f \in \soboH(\rot,\Omega)$ has a $\Lptan{2}(\Gamma)$ tangential trace, if there exists a $q \in \Lptan{2}(\Gamma)$ such that
    \begin{align*}
      \scprod{f}{\rot \phi}_{\Lp{2}(\Omega)} - \scprod{\rot f}{\phi}_{\Lp{2}(\Omega)} = \scprod{q}{\tanxtr \phi}_{\Lptan{2}(\Gamma)}
      \quad \forall \phi \in \Cc_{\partial\Omega \setminus \Gamma}(\R^{3}).
    \end{align*}
    In this case we say $\tantr f \in \Lptan{2}(\Gamma)$ and $\tantr f = q$ on $\Gamma$ or more precisely $\tantr[\Gamma] f = q$.

    \item
    We say $f \in \soboH(\rot,\Omega)$ has a $\Lptan{2}(\Gamma)$ twisted tangential trace, if there exists a $q \in \Lptan{2}(\Gamma)$ such that
    \begin{align*}
      \scprod{\rot f}{\phi}_{\Lp{2}(\Omega)} - \scprod{f}{\rot \phi}_{\Lp{2}(\Omega)} = \scprod{q}{\tantr \phi}_{\Lptan{2}(\Gamma)}
      \quad \forall \phi \in \Cc_{\partial\Omega \setminus \Gamma}(\R^{3}).
    \end{align*}
    In this case we say $\tanxtr f \in \Lptan{2}(\Gamma)$ and $\tanxtr f = q$ on $\Gamma$ or more precisely $\tanxtr[\Gamma] f = q$.
  \end{itemize}
\end{definition}

Note that the previous definition does not say anything about $\tantr f$ on $\partial\Omega \setminus \Gamma$. Moreover, the notation $\tantr[\Gamma]$ and $\tanxtr[\Gamma]$ is an abuse of notation, which should indicate that we only look at the part $\Gamma$ of the boundary $\partial\Omega$.

\begin{remark}
  Note that $\nu \times \tanxtr \phi = - \tantr \phi$ and $\scprod{q}{\tanxtr \phi}_{\Lptan{2}(\Gamma)} = \scprod{\nu \times q}{\nu \times \tanxtr \phi}_{\Lptan{2}(\Gamma)}$. Hence, we can easily see that $\tantr f \in \Lp{2}(\Gamma)$ is equivalent to $\tanxtr f \in \Lp{2}(\Gamma)$ and $\tanxtr f = \nu \times \tantr f$.
\end{remark}

\begin{definition}
  Let $\Omega$ be a (possibly unbounded) Lipschitz domain and $\Gamma \subseteq \partial\Omega$ open (in $\partial\Omega$).
  Then we define the space
  \begin{equation*}
    \hH_{\Gamma}(\rot, \Omega) \coloneqq
    \dset{f \in \soboH(\rot,\Omega)}{\tantr f \in \Lptan{2}(\Gamma)}
  \end{equation*}
  with its norm
  \begin{equation*}
    \norm{f}_{\hH_{\Gamma}(\rot, \Omega)} \coloneqq
    \Big(
    \norm{f}_{\Lp{2}(\Omega)}^{2} + \norm{\rot f}_{\Lp{2}(\Omega)}^{2}
    \mathclose{} + \norm{\tantr f}_{\Lp{2}(\Gamma)}^{2}
    \Big)^{\nicefrac{1}{2}}.
  \end{equation*}
  For $\Gamma = \partial\Omega$ we will just write $\hH(\rot,\Omega)$ instead of $\hH_{\partial\Omega}(\rot,\Omega)$.
\end{definition}

\begin{definition}
  Let $\Omega$ be a (possibly unbounded) Lipschitz domain and $\Gamma \subseteq \partial\Omega$ open (in $\partial\Omega$). Then we define the space
  \begin{equation*}
    \cH_{\Gamma}(\rot,\Omega) = \dset{f \in \hH_{\Gamma}(\rot,\Omega)}{\tantr[\Gamma] f = 0}.
  \end{equation*}
  For $\Gamma = \partial\Omega$ we will just write $\cH(\rot,\Omega)$ instead of $\cH_{\partial\Omega}(\rot,\Omega)$.
\end{definition}

In \cite[Thm.~4.5]{BaPaScho16} it is shown (for bounded $\Omega$) that $\Cc_{\Gamma}(\Omega)$ is dense in $\cH_{\Gamma}(\rot,\Omega)$ w.r.t. $\norm{\argdot}_{\soboH(\rot,\Omega)}$, i.e.,
\begin{equation*}
  \cH_{\Gamma}(\rot,\Omega) = \cl[\soboH(\rot,\Omega)]{\Cc_{\Gamma}(\Omega)}.
\end{equation*}
Hence, for homogeneous tangential traces there is already a version of the desired density result.

Note that the hat on top of the $\soboH$ indicates partial $\Lp{2}$ boundary conditions and the circle on top indicates partial homogeneous boundary conditions.

\begin{remark}
  We can immediately see
  \begin{equation*}
    \cH_{\Gamma}(\rot,\Omega) \subseteq \hH_{\Gamma}(\rot,\Omega).
  \end{equation*}
  Since $\tantr f \in \Lp{2}(\Gamma)$ is equivalent to $\tanxtr f \in \Lp{2}(\Gamma)$ we have
  \begin{equation*}
    \hH_{\Gamma}(\rot,\Omega) = \dset{f \in \soboH(\rot,\Omega)}{\tanxtr f \in \Lp{2}(\Gamma)},
  \end{equation*}
  Since $\tantr f = \tanxtr f \times \nu$, we have $\norm{\tantr f }_{\Lp{2}(\Gamma)} = \norm{\tanxtr f}_{\Lp{2}(\Gamma)}$ and
  \begin{equation*}
    \norm{f}_{\hH_{\Gamma}(\rot,\Omega)} =
    \big(
    \norm{f}_{\Lp{2}(\Omega)}^{2} + \norm{\rot f}_{\Lp{2}(\Omega)}^{2}
    + \norm{\tanxtr f}_{\Lp{2}(\Gamma)}^{2}
    \big)^{\nicefrac{1}{2}}.
  \end{equation*}
\end{remark}

\begin{remark}
  Since we use representation in an inner product, one can say that we have defined $\hH_{\Gamma}(\rot,\Omega)$ weakly. Another approach could have been to regard $\cl[\hH_{\Gamma}(\rot,\Omega)]{\Cc(\R^{3})}$, which could be called a strong approach. From this perspective the result we are going to show basically tells us that the weak and the strong approach to $\soboH(\rot,\Omega)$ fields that possess a $\Lptan{2}(\Gamma)$ tangential trace coincide.
\end{remark}


Recall the decomposition \eqref{eq:decomposition-of-cH-Gamma-0} for a Lipschitz pair $(\Omega,\Gamma_{0})$, where $\Omega$ is bounded:
\begin{equation*}
  \cH_{\Gamma_{0}}(\rot,\Omega) = \cH^{1}_{\Gamma_{0}}(\Omega) + \grad \cH_{\Gamma_{0}}^{1}(\Omega)
\end{equation*}
Note that every element in $\cH^{1}_{\Gamma_{0}}(\Omega)$ is already in $\hH(\rot,\Omega) \cap \cH_{\Gamma_{0}}(\rot,\Omega)$.
Moreover, by \cite[Lem.~3.1]{BaPaScho16} $\Cc_{\Gamma_{0}}(\R^{3})$ is dense in $\cH_{\Gamma_{0}}^{1}(\Omega)$ w.r.t.\ $\norm{\argdot}_{\soboH^{1}(\Omega)}$ and therefore also w.r.t.\ $\norm{\argdot}_{\hH(\rot,\Omega)}$.

Hence, it is left to show that every
\[
  f \in \grad \cH^{1}_{\Gamma_{0}}(\Omega) \cap \hH(\rot,\Omega)
\]
can be approximated by a $\Cc_{\Gamma_{0}}(\R^{3})$ function (w.r.t.\ $\norm{\argdot}_{\hH(\rot,\Omega)}$).

The following result seems obvious at first glance, but needs a bit of preparation to be shown. One must resist the temptation to show it for smooth function and conclude it by density, since this would lead to a circular argument. Luckily, it is basically \cite[Thm.~4.2]{Sk25a} and therefore already shown.

\begin{lemma}\label{le:H1-with-L2-tangrad-implies-W}
  Let $\Omega$ be a bounded Lipschitz domain, $\Gamma_{0} \subseteq \partial\Omega$ open and $f \in \cH_{\Gamma_{0}}^{1}(\Omega)$ such that $\grad f \in \hH(\rot,\Omega)$ (in particular $\tantr \grad f \in \Lptan{2}(\partial\Omega)$). Then $\tantr \grad f = \tangrad f\big\vert_{\partial\Omega}$ and $f \in W(\Omega) \cap \cH_{\Gamma_{0}}^{1}(\Omega)$.
\end{lemma}

\begin{proof}
  Since $\grad f \in \hH(\rot,\Omega)$, we know that $\tantr \grad f \in \Lp{2}(\partial\Omega)$ which implies
  $f \big\vert_{\partial\Omega} \in \soboH^{1}(\partial\Omega)$ and $\tangrad f \big\vert_{\partial\Omega} = \tantr \grad f$, see \cite[Thm.~4.2]{Sk25a}.
  Therefore, we conclude $f \in W(\Omega)$.
\end{proof}

This brings us to our first main theorem.

\begin{theorem}\label{th:Cc-gamma0-dense-in-hH-cap-cH-gamma0}
  Let $\Omega$ be bounded and $(\Omega,\Gamma_{0})$ a Lipschitz pair.
  Then
  $\Cc_{\Gamma_{0}}(\R^{3})$ is dense in $\hH(\rot,\Omega) \cap \cH_{\Gamma_{0}}(\rot,\Omega)$ w.r.t.\ $\norm{\argdot}_{\hH(\rot,\Omega)}$.
\end{theorem}

\begin{proof}
  Let $f \in \hH(\rot,\Omega) \cap \cH_{\Gamma_{0}}(\rot,\Omega)$ be arbitrary. Then we can decompose $f$ into $f = f_{1} + \grad f_{2}$, where $f_{1} \in \cH_{\Gamma_{0}}^{1}(\Omega)$ and $f_{2} \in \cH_{\Gamma_{0}}^{1}(\Omega)$. Note that $f \in \hH(\rot,\Omega)$ and $f_{1} \in \hH(\rot,\Omega)$, which implies $\grad f_{2} \in \hH(\rot,\Omega) \cap \grad \cH_{\Gamma_{0}}^{1}(\Omega)$.

  By \cite[Lem.~3.1]{BaPaScho16} $f_{1}$ can be approximated by $\Cc_{\Gamma_{0}}(\R^{3})$ functions w.r.t.\ $\norm{\argdot}_{\soboH^{1}(\Omega)}$ and therefore also w.r.t.\ $\norm{\argdot}_{\hH(\rot,\Omega)}$.

  By \Cref{le:H1-with-L2-tangrad-implies-W} we know that $f_{2} \in W(\Omega) \cap \cH^{1}_{\Gamma_{0}}(\Omega)$. Hence, we can apply \Cref{th:test-functions-dense-in-W} and obtain a sequence $(f_{2,n})_{n\in\N}$ that converges to $f_{2}$ w.r.t.\ $\norm{\argdot}_{W(\Omega)}$. This gives
  \begin{align*}
    \MoveEqLeft
    \norm{\grad f_{2} - \grad f_{2,n}}_{\hH(\rot,\Omega)}^{2} \\
    &= \norm{\grad (f_{2} - f_{2,n})}^{2}_{\Lp{2}(\Omega)}
      + \norm{\underbrace{\rot \grad (f_{2} - f_{2,n})}_{=\mathrlap{0}}}^{2}_{\Lp{2}(\Omega)}
      + \norm{\tantr \grad (f_{2} - f_{2,n})}^{2}_{\Lp{2}(\partial\Omega)} \\
    &\leq \norm{f_{2} - f_{2,n}}_{\soboH^{1}(\Omega)}^{2}
      + \norm[\big]{f_{2} \big\vert_{\partial\Omega} - f_{2,n} \big\vert_{\partial\Omega}}_{\soboH^{1}(\partial\Omega)}^{2} \\
    &= \norm{f_{2} - f_{2,n}}_{W(\Omega)}^{2} \to 0,
  \end{align*}
  which finishes the proof.
\end{proof}

\subsection{Density result without homogeneous part}%
\label{sec:density-without-homogeneous-part}

Since we already know that $\Cc_{\Gamma_{0}}(\R^{3})$ is dense in $\hH(\rot,\Omega) \cap \cH_{\Gamma_{0}}(\rot,\Omega)$, we can show the density of $\Cc(\R^{3})$ in $\hH_{\Gamma_{1}}(\rot,\Omega)$ by a duality argument, which we will present in this section.
This argument can be done in just a few lines within the notion of quasi Gelfand triples \cite{Sk24}. However, in order to stay relatively elementary we extract the essence and build a proof that avoids the introduction of this notion.

Basically we mimic the abstract boundary space for the tangential trace by $\cH(\rot,\Omega)^{\perp}$, which can also be viewed as the boundary space as it is isometrically isomorphic.

Our standing assumption in this section is that $\Omega$ is bounded, $(\Omega,\Gamma_{0})$ is a Lipschitz pair and $\Gamma_{1} \coloneqq \partial\Omega \setminus \cl{\Gamma_{0}}$.

Note that the boundedness of $\Omega$ is only necessary to apply \Cref{th:Cc-gamma0-dense-in-hH-cap-cH-gamma0}. However, we will later be able replace this by \Cref{th:Cc-gamma0-dense-in-hH-cap-cH-gamma0-unbounded}.

\begin{corollary}\label{th:characterization-of-hH-Gamma-with-closure-of-test-functions}
  If $f \in \hH_{\Gamma_{1}}(\rot,\Omega)$, then
  \begin{align*}
    \scprod{\tanxtr f}{\tantr g}_{\Lp{2}(\Gamma_{1})} = \scprod{\rot f}{g}_{\Lp{2}(\Omega)} - \scprod{f}{\rot g}_{\Lp{2}(\Omega)}
  \end{align*}
  for all $g \in \hH(\rot,\Omega) \cap \cH_{\Gamma_{0}}(\rot,\Omega)$.
\end{corollary}

\begin{proof}
  For $f \in \hH_{\Gamma_{1}}(\rot,\Omega)$ we have by definition
  \begin{equation*}
    \scprod{\tanxtr f}{\tantr g}_{\Lp{2}(\Gamma_{1})} = \scprod{\rot f}{g}_{\Lp{2}(\Omega)} - \scprod{f}{\rot g}_{\Lp{2}(\Omega)}
  \end{equation*}
  for all $g \in \Cc_{\Gamma_{0}}(\R^{3})$. Since this equation is continuous in $g$ w.r.t.\ $\norm{\argdot}_{\hH(\rot,\Omega)}$, we can extend it by continuity to $g \in \cl[\hH(\rot,\Omega)]{\Cc_{\Gamma_{0}}(\R^{3})}$ and by \Cref{th:Cc-gamma0-dense-in-hH-cap-cH-gamma0} to $g \in \hH(\rot,\Omega) \cap \cH_{\Gamma_{0}}(\rot,\Omega)$.
\end{proof}


\begin{lemma}\label{le:cH-rot-perp}
  We have the following identity
  \begin{equation*}
    \cH(\rot,\Omega)^{\perp} = \dset{f \in \soboH(\rot,\Omega)}{\rot\rot f = -f},
  \end{equation*}
  where the orthogonal complement is taken in $\soboH(\rot,\Omega)$, i.e., w.r.t.\ $\scprod{\argdot}{\argdot}_{\soboH(\rot,\Omega)}$.
  Moreover, $\rot$ leaves the space $\cH(\rot,\Omega)^{\perp}$ invariant.
\end{lemma}

\begin{proof}
  Note that by density of $\Cc(\Omega)$ in $\cH(\rot,\Omega)$ both spaces have the same orthogonal complement. Hence,
  \begin{align*}
    f \in \cH(\rot,\Omega)^{\perp}
    \quad&\Leftrightarrow\quad
    0 = \scprod{f}{g}_{\Lp{2}(\Omega)} + \scprod{\rot f}{\rot g}_{\Lp{2}(\Omega)}
    \quad \forall\, g \in \Cc(\Omega) \\
    \quad&\Leftrightarrow\quad \rot f \in \soboH(\rot,\Omega) \quad\text{and}\quad \rot \rot f = - f. \qedhere
  \end{align*}

\end{proof}

\begin{lemma}\label{le:ortho-projection-on-cH-perp}
  Let $P$ be the orthogonal projection on $\cH(\rot,\Omega)^{\perp}$ (in $\soboH(\rot,\Omega)$).
  Then $\hH(\rot,\Omega) \cap \cH_{\Gamma_{0}}(\rot,\Omega)$ is invariant under $P$, i.e., $f \in \hH(\rot,\Omega) \cap \cH_{\Gamma_{0}}(\rot,\Omega)$ implies $Pf \in \hH(\rot,\Omega) \cap \cH_{\Gamma_{0}}(\rot,\Omega)$.
\end{lemma}

\begin{proof}
  Since $\idop - P$ is the orthogonal projection on $\cH(\rot,\Omega)$ and $\cH(\rot,\Omega)$ is a subspace of $\hH(\rot,\Omega) \cap \cH_{\Gamma_{0}}(\rot,\Omega)$, we conclude that $(\idop - P)f \in \hH(\rot,\Omega) \cap \cH_{\Gamma_{0}}(\rot,\Omega)$ for every $f \in \soboH(\rot,\Omega)$. Now for every $f \in \hH(\rot,\Omega) \cap \cH_{\Gamma_{0}}(\rot,\Omega)$ we have
  \begin{equation*}
    P f = f - (\idop - P)f,
  \end{equation*}
  which is in $\hH(\rot,\Omega) \cap \cH_{\Gamma_{0}}(\rot,\Omega)$, as linear combination of $\hH(\rot,\Omega) \cap \cH_{\Gamma_{0}}(\rot,\Omega)$ elements.
\end{proof}

\begin{lemma}\label{le:range-of-tantr-characterizations}
  For every $q \in \tantr\big( \hH(\rot,\Omega) \cap \cH_{\Gamma_{0}}(\rot,\Omega) \big)$ there exists a $g \in \cH(\rot,\Omega)^{\perp}$ such that
  \begin{equation*}
    \rot g \in \hH(\rot,\Omega) \cap \cH_{\Gamma_{0}}(\rot,\Omega) \cap \cH(\rot,\Omega)^{\perp}
    \quad\text{and}\quad \tantr \rot g = q.
  \end{equation*}
  In particular,
  \[
    \tantr\big( \hH(\rot,\Omega) \cap \cH_{\Gamma_{0}}(\rot,\Omega) \big) = \tantr\big( \hH(\rot,\Omega) \cap \cH_{\Gamma_{0}}(\rot,\Omega) \cap \cH(\rot,\Omega)^{\perp}\big).
  \]
\end{lemma}

\begin{proof}
  By assumption we have $q = \tantr f$ for $f \in \hH(\rot,\Omega) \cap \cH_{\Gamma_{0}}(\rot,\Omega)$. Let $P$ denote the orthogonal projection on $\cH(\rot,\Omega)^{\perp}$. Then by \Cref{le:ortho-projection-on-cH-perp} we can decompose $f$ into $f = Pf + (\idop - P)f$, where both $Pf$ and $(\idop - P)f$ are also in $\hH(\rot,\Omega) \cap \cH_{\Gamma_{0}}(\rot,\Omega)$. Moreover, $(\idop - P)f \in \cH(\rot,\Omega)$, which gives $\tantr (\idop - P)f = 0$ and therefore
  \begin{equation*}
    q = \tantr f = \tantr Pf.
  \end{equation*}
  Since $Pf \in \cH(\rot,\Omega)^{\perp}$, we have $\rot \rot Pf = -Pf$. Thus defining $g = -\rot Pf$ finishes the proof.
\end{proof}

Now we finally come to our second main theorem.

\begin{theorem}\label{th:density-in-hH-gamma1}
  $\Cc(\R^{3})$ is dense in $\hH_{\Gamma_{1}}(\rot,\Omega)$ w.r.t.\ $\norm{\argdot}_{\hH_{\Gamma_{1}}(\rot,\Omega)}$.
\end{theorem}

\begin{proof}
  \newcommand{\restanxtr}{\widetilde{\gamma}_{\tau}}
  By the definition of the norm of $\hH_{\Gamma_{1}}(\rot,\Omega)$ the mapping $\tanxtr \colon \hH_{\Gamma_{1}}(\rot,\Omega) \subseteq \soboH(\rot,\Omega) \to \Lptan{2}(\Gamma_{1})$ is closed. We define the following restriction of $\tanxtr$
  \begin{equation*}
    \restanxtr \colon
    \mapping{\Cc(\R^{3}) \subseteq \soboH(\rot,\Omega)}{\Lptan{2}(\Gamma_{1})}{f}{\tanxtr f.}
  \end{equation*}
  Since $\restanxtr \subseteq \tanxtr$ we conclude
  \begin{equation*}
    \restanxtr\adjun \supseteq \tanxtr\adjun.
  \end{equation*}
  \begin{steps}
    \item \emph{Calculate $\dom \restanxtr\adjun$.}
    Let $q \in \dom \restanxtr\adjun$. Then there exists a $g \in \soboH(\rot,\Omega)$ such that
    \begin{equation}\label{eq:adjoint-formulation-for-hat-tanxtr}
      \scprod{\restanxtr f}{q}_{\Lp{2}(\Gamma_{1})} = \scprod{f}{g}_{\soboH(\rot,\Omega)} = \scprod{f}{g}_{\Lp{2}(\Omega)} + \scprod{\rot f}{\rot g}_{\Lp{2}(\Omega)}
    \end{equation}
    for all $f \in \Cc(\R^{3})$. For $f \in \Cc_{\Gamma_{1}}(\R^{3})$, we obtain
    \begin{equation*}
      0 = \scprod{f}{g}_{\Lp{2}(\Omega)} + \scprod{\rot f}{\rot g}_{\Lp{2}(\Omega)},
    \end{equation*}
    which implies $\rot g \in \cH_{\Gamma_{0}}(\rot,\Omega)$ and $\rot\rot g = -g$, and by \Cref{le:cH-rot-perp} $g \in \cH(\rot,\Omega)^{\perp}$. Hence, we revisit~\eqref{eq:adjoint-formulation-for-hat-tanxtr}, where we extend $q$ by $0$ outside of $\Gamma_{1}$ in $\partial\Omega$ and use the full twisted tangential trace $\tanxtr$ (on all of $\partial\Omega$) for $f \in \Cc(\R^{3})$
    \begin{equation*}
      \scprod{\tanxtr f}{q}_{\Lp{2}(\partial\Omega)} = -\scprod{f}{\rot \rot g}_{\Lp{2}(\Omega)} + \scprod{\rot f}{\rot g}_{\Lp{2}(\Omega)}.
    \end{equation*}
    This implies $\rot g \in \hH(\rot,\Omega)$ and $q = \tantr \rot g$.
    Consequently,
    \begin{align*}
      \dom \restanxtr\adjun
      &\subseteq \tantr\big( \hH(\rot,\Omega) \cap \cH_{\Gamma_{0}}(\rot,\Omega) \cap \cH(\rot,\Omega)^{\perp} \big) \\
      &= \tantr\big( \hH(\rot,\Omega) \cap \cH_{\Gamma_{0}}(\rot,\Omega) \big).
    \end{align*}

    \item \emph{Calculate $\dom \tanxtr\adjun$.}
    %
    If $q \in \tantr\big( \hH(\rot,\Omega) \cap \cH_{\Gamma_{0}}(\rot,\Omega) \big)$, then by \Cref{le:range-of-tantr-characterizations} there exists a $g \in \cH(\rot,\Omega)^{\perp}$ such that $\rot g \in \hH(\rot,\Omega) \cap \cH_{\Gamma_{0}}(\rot,\Omega)$ and $\tantr \rot g = q$. Hence, by \Cref{th:characterization-of-hH-Gamma-with-closure-of-test-functions} for $f \in \hH_{\Gamma_{1}}(\rot,\Omega)$ and $\rot g$ we have
    \begin{equation*}
      \scprod{\tanxtr f}{\underbrace{\tanxtr \rot g}_{=\mathrlap{q}}}_{\Lp{2}(\Gamma_{1})} = \scprod{\rot f}{\rot g}_{\Lp{2}(\Omega)} - \scprod{f}{\underbrace{\rot \rot g}_{=\mathrlap{-g}}}_{\Lp{2}(\Omega)}
      = \scprod{f}{g}_{\soboH(\rot,\Omega)},
    \end{equation*}
    which implies $q \in \dom \tanxtr\adjun$.
    Consequently,
    \[
      \dom \tanxtr\adjun \supseteq \tantr\big( \hH(\rot,\Omega) \cap \cH_{\Gamma_{0}}(\rot,\Omega) \big).
    \]

    \item Combining the results of the previous steps yields
    \begingroup%
    \newcommand{\tmpEq}{%
      \tantr\big( \hH(\rot,\Omega) \cap \cH_{\Gamma_{0}}(\rot,\Omega) \big)
      \supseteq \dom \restanxtr\adjun \\%
      \supseteq \dom \tanxtr\adjun
      \supseteq \tantr\big( \hH(\rot,\Omega) \cap \cH_{\Gamma_{0}}(\rot,\Omega) \big).%
    }%
        \begin{multline*}
          \tmpEq
        \end{multline*}%
    \endgroup%
    Hence, $\restanxtr\adjun = \tanxtr\adjun$ and therefore
    \begin{align*}
      \cl{\restanxtr} = \restanxtr\adjun[2] = \tanxtr\adjun[2] = \tanxtr,
    \end{align*}
    which implies $\Cc(\R^{3})$ is dense in $\hH_{\Gamma_{1}}(\rot,\Omega)$ w.r.t.\ the graph norm of $\tanxtr$ which is $\norm{\argdot}_{\hH_{\Gamma_{1}}(\rot,\Omega)}$.
    \qedhere
  \end{steps}
\end{proof}

\subsection{Integration by parts}

Let $\Omega \subseteq \R^{3}$ be a Lipschitz domain and $\Gamma_{0},\Gamma_{1} \subseteq \partial\Omega$ a decomposition of $\partial\Omega$ such that $(\Omega,\Gamma_{0})$ is a Lipschitz pair.
The integration by parts formula
\begin{equation*}
  \scprod{\rot f}{g}_{\Lp{2}(\Omega)} - \scprod{f}{\rot g}_{\Lp{2}(\Omega)} = \scprod{\tanxtr f}{\tantr g}_{\Lp{2}(\Gamma_{1})}
\end{equation*}
is a priori valid for $f \in \Cc(\R^{3})$ and $g \in \Cc_{\Gamma_{0}}(\R^{3})$.
Moreover, by definition of $\hH_{\Gamma_{1}}(\rot,\Omega)$ and $\hH_{\partial\Omega}(\rot,\Omega) \cap \cH_{\Gamma_{0}}(\rot,\Omega)$ we can extend this to either $f \in \hH_{\Gamma_{1}}(\rot,\Omega)$ or $g \in \hH_{\partial\Omega}(\rot,\Omega) \cap \cH_{\Gamma_{0}}(\rot,\Omega)$, but a priori not to both simultaneously.

As we have seen in \Cref{th:characterization-of-hH-Gamma-with-closure-of-test-functions} one density result is already sufficient to show that a simultaneous extension is valid. However, with both density results we conclude that this extension is even a continuous extension.

\section{The unbounded case}%
\label{sec:unbounded-domains}
In this section we consider an unbounded Lipschitz domain $\Omega \subseteq \R^{3}$ such that $(\Omega,\Gamma_{0})$ is a Lipschitz pair for $\Gamma_{0} \subseteq \partial\Omega$ and we define $\Gamma_{1} \coloneqq \partial\Omega \setminus \cl{\Gamma_{0}}$.

The main idea is to multiply a vector field $f \in \hH_{\Gamma_{1}}(\rot,\Omega)$ with a cutoff function $\chi$ that is supported on a sufficiently large ball such that $\norm{f - \chi f}_{\hH_{\Gamma_{1}}(\rot,\Omega)} \leq \epsilon$ for given $\epsilon > 0$. Then we only have to approximate $\chi f$ with smooth vector fields. This reduces the problem again to bounded domains, as the support of $\chi f$ is bounded.

In order to execute these ideas we show some technical lemmas, that explain how the boundary conditions are inherited in a smooth cutoff process. Intuitively the lemmas are obvious, however we present them for completeness.

\begin{lemma}\label{le:mult-with-cut-of-stays-in-hH}
  Let $f \in \hH_{\Gamma_{1}}(\rot,\Omega)$ and $\chi \in \Cc(\R^{3})$.
  Then $\chi f \in \hH_{\Gamma_{1}}(\rot,\Omega)$ with $\tantr (\chi f) = \chi \tantr f$ on $\Gamma_{1}$.
\end{lemma}

\begin{proof}
  Let $\phi \in \Cc_{\Gamma_{0}}(\R^{3})$. Note that by the product rule for $\rot$ we have $\rot (\chi f) = \chi \rot f + \grad \chi \times f$. Moreover, we have $\chi \phi \in \Cc_{\Gamma_{0}}(\R^{3})$ and we denote $\tantr f$ as $q \in \Lptan{2}(\Gamma_{1})$ on $\Gamma_{1}$.
  Hence,
  \begin{align*}
    \MoveEqLeft
    \scprod{\phi}{\rot (\chi f)}_{\Lp{2}(\Omega)} \\
    &= \scprod{\phi}{\chi \rot f}_{\Lp{2}(\Omega)} + \scprod{\phi}{\grad \chi \times f}_{\Lp{2}(\Omega)}
      = \scprod{\chi \phi}{\rot f}_{\Lp{2}(\Omega)} - \scprod{\grad \chi \times \phi}{f}_{\Lp{2}(\Omega)} \\
    &=  \scprod{\rot (\chi \phi)}{f}_{\Lp{2}(\Omega)} + \scprod{\tanxtr (\chi \phi)}{q}_{\Lptan{2}(\Gamma_{1})} - \scprod{\grad \chi \times \phi}{f}_{\Lp{2}(\Omega)} \\
    &= \scprod{\chi \rot \phi}{f}_{\Lp{2}(\Omega)} + \scprod{\tanxtr \phi}{\chi q}_{\Lptan{2}(\Gamma_{1})},
  \end{align*}
  which implies $\chi f \in \hH_{\Gamma_{1}}(\rot,\Omega)$ with $\tantr \chi f = \chi q$.
\end{proof}

\begin{corollary}
  Let $f \in \hH_{\partial\Omega}(\rot,\Omega) \cap \cH_{\Gamma_{0}}(\rot,\Omega)$ and $\chi \in \Cc(\R^{3})$.
  Then $\chi f \in \hH_{\partial\Omega}(\rot,\Omega) \cap \cH_{\Gamma_{0}}(\rot,\Omega)$ with $\tantr (\chi f) = \chi \tantr f$ on $\Gamma_{1}$.
\end{corollary}

\begin{proof}
  By \Cref{le:mult-with-cut-of-stays-in-hH} applied with $\Gamma_{1} = \partial\Omega$ we have $\chi f \in \hH_{\partial\Omega}(\rot,\Omega)$ and $\tantr \chi f = \chi \tantr f$. Since $\tantr f = 0$ on $\Gamma_{0}$ we conclude that also $\tantr \chi f = 0$ on $\Gamma_{0}$ and consequently $\chi f \in \cH_{\Gamma_{0}}(\rot,\Omega)$.
\end{proof}

\begin{lemma}\label{le:cutoff-function-in-cylinder}
  Let $f \in \hH(\rot,\Omega) \cap \cH_{\Gamma_{0}}(\rot,\Omega)$, $p \in \partial\Omega$, $\alpha \in \Cc(C_{\epsilon,h}(p))$, $\Omega_{p} = \Omega \cap C_{\epsilon,h}(p)$ and $\Gamma_{1,p} = \Gamma_{1} \cap \supp \alpha$.
  Then $\alpha f \in \hH(\rot,\Omega_{p}) \cap \cH_{\partial\Omega_{p} \setminus \cl{\Gamma_{1,p}}}(\rot,\Omega_{p})$ with $\tantr \alpha f = \alpha \tantr f$ on $\partial\Omega_{p}$, where $\tantr f$ is extended by $0$ on $\partial\Omega_{p} \setminus \partial\Omega$.
\end{lemma}

The setting of the previous lemma is illustrated in \Cref{fig:cutoff-function-in-cylinder}.

\begin{figure}[ht]
  \centering
  \begin{tikzpicture}[scale=2]
  \coordinate (p) at (0.5,1);

  \draw (1,0) .. controls (0.48,0.18) .. (p);
  \draw[name path=abschnitt2] (p) .. controls (1,1.7) .. (2,2);
  \draw
  (2,2) to [out=0,in=190] (3,2.2);
  \draw[dashed] (3,2.2) to[out=10,in=190] (4,2.5);

  \path [name path=rect, draw=none] (-1,1.9) -- (3,1.9);
  \path [name intersections={of = abschnitt2 and rect, by=p2}];

  \draw ($(p2) + (-0.8,-0.8)$) rectangle ($(p2) + (0.8,0.8)$);
  \node[anchor=north west] at ($(p2) + (-0.8,0.8)$) {$C_{\epsilon,h}(p)$};

  \draw[fill=yellow,opacity=0.2]
  ($(p2) + (0,-0.6)$)
  to[out=170,in=-100] ($(p2) + (-0.7,0)$)
  to[out=80,in=200] ($(p2) + (0,0.6)$)
  to[out=20,in=90] ($(p2) + (0.7,0.4)$)
  to ($(p2) + (0.7,0)$)
  to[out=-90,in=-10] ($(p2) + (0,-0.6)$);

  \begin{scope}
    \clip (1,0) .. controls (0.48,0.18) .. (p) .. controls (1,1.7) .. (2,2)
    to [out=0,in=190] (3,2.2) |- (1.5,0.3) -- cycle;



    \fill[green,opacity=0.2] ($(p2) + (-0.8,-0.8)$) rectangle ($(p2) + (0.8,0.8)$);
  \end{scope}

  \node[green!50!black, anchor=south west] at ($(p2) + (-0.75,-0.75)$) {$\Omega_{p}$};
  \node[yellow!40!black] at ($(p2) + (0,0.25)$) {$\supp \alpha$};
  \node[blue,left] at (p) {$\Gamma_{1}$};
  \node[red,above] at (3,2.2) {$\Gamma_{0}$};

  \begin{scope}
    \clip (1,0) .. controls (0.48,0.18) .. (p) .. controls (1,1.7) .. (2,2)
    to [out=0,in=190] (3,2.2) |- (1.5,0.3) -- cycle;

    \draw[very thick,green!50!black] ($(p2) + (-0.8,-0.8)$) rectangle ($(p2) + (0.8,0.8)$);
  \end{scope}

  \begin{scope}
    \clip ($(p2) + (-0.8,-0.8)$) rectangle ($(p2) + (0.8,0.8)$);
    \clip (1,0) .. controls (0.48,0.18) .. (p) .. controls (1,1.7) .. (2,2)
    to [out=0,in=190] (3,2.2) |- (1.5,0.3) -- cycle;

    \draw[line width=1.5mm,green!50!black] (1,0) .. controls (0.48,0.18) .. (p) .. controls (1,1.7) .. (2,2);
    \draw[line width=1.5mm,green!50!black] (2,2) to [out=0,in=190] (3,2.2);

  \end{scope}

  \draw[very thick,blue] (1,0) .. controls (0.48,0.18) .. (p) .. controls (1,1.7) .. (2,2);
  \draw[very thick,blue] (1,0) to  [out=10,in = 180] (3,0.2);
  \draw[very thick,blue,dashed] (3,0.2) to [out=0,in=170] (4,0);
  \draw[very thick,red]
  (2,2) to [out=0,in=190] (3,2.2);
  \draw[very thick, red, dashed] (3,2.2) to[out=10,in=190] (4,2.5);

  \filldraw (p2) circle (0.025);
  \node[below] at (p2) {$p$};

  \node at (1.7,0.8) {$\Omega$};

\end{tikzpicture}

  \caption{\label{fig:cutoff-function-in-cylinder}Illustration of the setting of \Cref{le:cutoff-function-in-cylinder}}
\end{figure}

\begin{proof}
  Note that $\alpha f \in \hH(\rot,\Omega) \cap \cH_{\Gamma_{0}}(\rot,\Omega)$ by \Cref{le:mult-with-cut-of-stays-in-hH}. Hence, for $\phi \in \Cc(\R^{3})$ we have
  \begin{align*}
    \MoveEqLeft
    \scprod{\rot \phi}{\alpha f}_{\Lp{2}(\Omega_{p})} - \scprod{\phi}{\rot \alpha f}_{\Lp{2}(\Omega_{p})} \\
    &= \scprod{\rot \phi}{\alpha f}_{\Lp{2}(\Omega)} - \scprod{\phi}{\rot \alpha f}_{\Lp{2}(\Omega)}
    = \scprod{\tanxtr \phi}{\alpha \tantr f}_{\Lp{2}(\Gamma_{1})} \\
    &= \scprod{\tanxtr \phi}{\alpha \tantr f}_{\Lp{2}(\Gamma_{1,p})}
    = \scprod{\tanxtr \phi}{\alpha \tantr f}_{\Lp{2}(\partial\Omega_{p})}
  \end{align*}
  Hence, $\alpha f\in \hH(\rot,\Omega_{p})$. Moreover, if we choose $\phi \in \Cc_{\cl{\Gamma_{1,p}}}(\R^{3})$ in the previous equation we conclude $\alpha f \in \cH_{\partial\Omega \setminus \cl{\Gamma_{1,p}}}(\rot,\Omega_{p})$.
\end{proof}

\begin{remark}
  Note that in general the intersection of Lipschitz domains is not a Lipschitz domain. Roughly speaking this is because of the angle between the sets. However, for $p \in \partial\Omega$ the intersection $\Omega_{p} \coloneqq C_{\epsilon,h}(p) \cap \Omega$ of the cylinder $C_{\epsilon,h}(p)$ and $\Omega$ is again a Lipschitz domain. Moreover, $(\Omega_{p},\Gamma_{0} \cap \Omega_{p})$ is a Lipschitz pair. Roughly speaking here the assumptions on $C_{\epsilon,h}(p)$ make sure that angle is not $0$.
\end{remark}

Finally, we show that density result for unbounded domains.

\begin{theorem}\label{th:Cc-gamma0-dense-in-hH-cap-cH-gamma0-unbounded}
  Let $\Omega$ be an (unbounded) Lipschitz domain such that $(\Omega,\Gamma_{0})$ is a Lipschitz pair.
  Then $\Cc_{\Gamma_{0}}(\R^{3})$ is dense in $\hH_{\partial\Omega}(\rot,\Omega) \cap \cH_{\Gamma_{0}}(\rot,\Omega)$ w.r.t.\ $\norm{\argdot}_{\hH_{\partial\Omega}(\rot,\Omega)}$.
\end{theorem}

\begin{proof}
  Let $f \in \hH_{\partial\Omega}(\rot,\Omega) \cap \cH_{\Gamma_{0}}(\rot,\Omega)$.
  \begin{steps}
    \item \emph{Reduce to bounded case.} Then for given $\epsilon$ there exists (by the monotone convergence theorem) an $r > 0$ such that
    \begin{align*}
      \norm{f - \indicator_{\ball_{r}(0)} f}_{\Lp{2}(\Omega)} &\leq \epsilon, \\
      \norm{\rot f - \indicator_{\ball_{r}(0)} \rot f}_{\Lp{2}(\Omega)} &\leq \epsilon,\\
      \norm{\tantr f - \indicator_{\ball_{r}(0)} \tantr f}_{\Lp{2}(\partial\Omega)} &\leq \epsilon.
    \end{align*}
    This is a consequence of monotone convergence. We choose $\chi \in \Cc(\R^{3})$ such that
    \begin{equation*}
      \supp \chi \subseteq \ball_{r+2}(0),\quad \chi(\zeta) \in [0,1],\quad \chi = 1\ \text{on}\ \ball_{r}(0) \quad\text{and}\quad \norm{\grad \chi}_{\infty} \leq 1.
    \end{equation*}
    Such a function can be constructed by convolving $\indicator_{\ball_{r}(0)}$ with a suitable mollifier.
    By \Cref{le:mult-with-cut-of-stays-in-hH} we have $\chi f \in \hH_{\partial\Omega}(\rot,\Omega) \cap \cH_{\Gamma_{0}}(\rot,\Omega)$ with $\tantr \chi f = \chi \tantr f$.
    Moreover, by the triangle inequality and $\chi = 1$ on $\ball_{r}(0)$ we have
    \begin{align*}
      \norm{f - \chi f}_{\Lp{2}(\Omega)} &\leq 2\epsilon, \\
      \norm{\rot f - \rot \chi f}_{\Lp{2}(\Omega)} &\leq \norm{\rot f - \chi \rot f}_{\Lp{2}(\Omega)} + \norm{\grad \chi \times f}_{\Lp{2}(\Omega)} \leq 4 \epsilon,\\
      \norm{\tantr f - \chi \tantr f}_{\Lp{2}(\partial\Omega)} &\leq 2\epsilon.
    \end{align*}
    Hence, $\norm{f - \chi f}_{\hH_{\partial\Omega}(\rot,\Omega)} \leq 8 \epsilon$.

    \item \emph{Localize by partition of unity.} We define $O_{p} \coloneqq C_{\epsilon,h}(p)$ for $p \in \partial\Omega \cap \ball_{r+2}(0)$, where $C_{\epsilon,h}(p)$ is a cylinder that fits the Lipschitz assumptions of $\Omega$ and $O_{p} \coloneqq \ball_{\epsilon}(p)$ for $p \in \Omega \cap \ball_{r+2}(0)$, where $\epsilon > 0$ is such that $\ball_{\epsilon}(p) \subseteq \Omega$. Then $(O_{p})_{p \in (\partial\Omega \cup \Omega) \cap \ball_{r+2}(0)}$ is a covering of $\cl{\Omega \cap \ball_{r+2}(0)}$. Since $\cl{\Omega \cap \ball_{r+2}(0)}$ is compact there exists a finite subcovering $(O_{i})_{i=1}^{k}$.
    We define $\Omega_{i} = O_{i} \cap \Omega$.
    Moreover, we employ a partition of unity subordinate to this covering and obtain functions $(\alpha_{i})_{i=1}^{k}$ such that
    \begin{equation*}
      \alpha_{i} \in \Cc(O_{i}), \quad \alpha_{i}(\zeta) \in [0,1],\quad\text{and}\quad \sum_{i=1}^{k} \alpha_{i}(\zeta) = 1 \quad \text{for}\quad \zeta \in \cl{\Omega \cap \ball_{r+2}(0)}
    \end{equation*}
    We define $\Gamma_{1,i} \coloneqq \Gamma_{1} \cap \supp \alpha_{i}$.
    Hence, if $O_{i}$ is a cylinder, then $\alpha_{i}\chi f \in \hH_{\partial\Omega_{i}}(\rot,\Omega_{i}) \cap \cH_{\partial\Omega_{i} \setminus \cl{\Gamma_{1,i}}}(\rot,\Omega_{i})$ and if $O_{i}$ is a ball, then $\alpha_{i} \chi f \in \cH_{\partial\Omega_{i}}(\rot,\Omega_{i})$.
    In both cases we can extend these functions to $\Omega$ by $0$ outside of $\Omega_{i}$ and obtain an element of $\hH_{\partial\Omega}(\rot,\Omega) \cap \cH_{\Gamma_{0}}(\rot,\Omega)$.

    Moreover, we can approximate these functions by $\Cc_{\partial\Omega_{i} \setminus \cl{\Gamma_{1,i}}}(\R^{3})$ functions (that vanish on $\Omega \setminus \Omega_{i}$) w.r.t.\ $\norm{\argdot}_{\hH(\rot,\Omega_{i})}$ by \Cref{th:Cc-gamma0-dense-in-hH-cap-cH-gamma0}.
    Since the extended versions are $0$ outside of $\Omega_{i}$ we conclude this approximation also in $\hH_{\partial\Omega}(\rot,\Omega)$.
    Hence, we can approximate every $\alpha_{i}\chi f$ by a sequence in $\Cc_{\Gamma_{0}}(\R^{3})$ (w.r.t.\ $\norm{\argdot}_{\hH_{\partial\Omega}(\rot,\Omega)}$).

    We denote these sequences by $(f_{i,n})_{n\in\N}$ and we choose $n \in \N$ so large that $\norm{\alpha_{i}\chi f - f_{i,n}}_{\hH_{\partial\Omega}(\rot,\Omega)} \leq \frac{\epsilon}{k}$. Since $\chi f = 0$ outside of $\ball_{r+2}(0)$, we have $\sum_{i=1}^{k}\alpha_{i}\chi f = \chi f$ and therefore
    \begin{align*}
      \norm[\Big]{f - \sum_{i=1}^{k}f_{i,n}}_{\hH_{\partial\Omega}(\rot,\Omega)}
      &\leq \norm{f - \chi f}_{\hH_{\partial\Omega}(\rot,\Omega)}
        + \sum_{i=1}^{k} \norm{\alpha_{i}\chi f - f_{i,n}}_{\hH_{\partial\Omega}(\rot,\Omega)} \\
      &\leq 8\epsilon + \epsilon. \qedhere
    \end{align*}
  \end{steps}
\end{proof}

\begin{corollary}
  Let $\Omega$ be an (unbounded) Lipschitz domain. Then $\Cc(\Omega)$ is dense in $\cH(\rot,\Omega)$ w.r.t.\ $\norm{\argdot}_{\soboH(\rot,\Omega)}$.
\end{corollary}

\begin{proof}
  This is just a special case of \Cref{th:Cc-gamma0-dense-in-hH-cap-cH-gamma0-unbounded}, namely $\Gamma_{0} = \partial\Omega$.
\end{proof}

In order to generalize the other main result for unbounded domains we have two options: Either repeating the strategy of this section with small adaptions or repeating \Cref{sec:density-without-homogeneous-part}. The second option is more convenient as we only need to observe that boundedness was only used for \Cref{th:Cc-gamma0-dense-in-hH-cap-cH-gamma0}, which we can now replace by \Cref{th:Cc-gamma0-dense-in-hH-cap-cH-gamma0-unbounded}. Hence, we obtain the following theorem.

\begin{theorem}\label{th:density-in-hH-gamma1-unbounded}
  Let $\Omega$ be an (unbounded) Lipschitz domain such that $(\Omega,\Gamma_{1})$ is a Lipschitz pair.
  Then $\Cc(\R^{3})$ is dense in $\hH_{\Gamma_{1}}(\rot,\Omega)$ w.r.t.\ $\norm{\argdot}_{\hH_{\Gamma_{1}}(\rot,\Omega)}$.
\end{theorem}

\begin{proof}
  Repeat \Cref{sec:density-without-homogeneous-part} and replace \Cref{th:Cc-gamma0-dense-in-hH-cap-cH-gamma0} with \Cref{th:Cc-gamma0-dense-in-hH-cap-cH-gamma0-unbounded}.
\end{proof}

\begin{corollary}
  Let $\Omega$ be an (unbounded) Lipschitz domain. Then $\Cc(\R^{3})$ is dense in $\soboH(\rot,\Omega)$ w.r.t.\ $\norm{\argdot}_{\soboH(\rot,\Omega)}$.
\end{corollary}

\begin{proof}
  This is just a special case of \Cref{th:density-in-hH-gamma1-unbounded}, namely $\Gamma_{1} = \emptyset$.
\end{proof}

\section{Conclusion}%
\label{sec:conclusion}

We have defined $\soboH(\rot,\Omega)$ fields that possess an $\Lp{2}$ tangential trace on $\Gamma_{1} \subseteq \partial\Omega$ via a weak approach (by representation in the $\Lp{2}(\Gamma_{1})$ inner product) and showed that the $\conC^{\infty}$ fields are dense in this space. This is a generalization of \cite{BeBeCoDa97}, where the case $\Gamma_{1} = \partial\Omega$ was regarded. Moreover, we lifted the result to unbounded domains. In fact for partial tangential traces there is the second question about the density with additional homogeneous boundary conditions on $\Gamma_{0} = \partial\Omega \setminus \cl{\Gamma_{1}}$.
This was exactly the open problem in \cite[Sec.~5]{WeSt13}, which we could solve.
In particular they were asking whether $\soboH_{\Gamma_{0}}^{1}(\Omega)$ is dense in $\hH(\rot,\Omega) \cap \cH_{\Gamma_{0}}(\rot,\Omega)$, which is in fact a weaker version of \Cref{th:Cc-gamma0-dense-in-hH-cap-cH-gamma0}. We answered this question even for unbounded domains with \Cref{th:Cc-gamma0-dense-in-hH-cap-cH-gamma0-unbounded}.


%


\end{document}